\documentclass[a4paper,10pt,notitlepage,reqno]{article}
\usepackage[english]{babel}
\usepackage{amssymb,amsthm,amsmath} 
\usepackage{mathtools}
\usepackage{mathrsfs}
\usepackage{enumerate}
\usepackage{authblk}
\usepackage[noadjust]{cite}
\usepackage[symbol]{footmisc}
\usepackage{soul}
\usepackage[dvipsnames]{xcolor}
\usepackage{geometry}
\usepackage[colorlinks=true, linkcolor=red, citecolor=blue, urlcolor=blue,hyperfootnotes=false]{hyperref}
\usepackage{cleveref}

\theoremstyle{plain}
\newtheorem{theorem}{Theorem}[section]
\newtheorem{lemma}{Lemma}[section]

\newtheorem{corollary}{Corollary}[section]

\theoremstyle{definition}

\theoremstyle{remark}
\newtheorem{remark}{Remark}[section]

\DeclarePairedDelimiter{\abs}{\lvert}{\rvert} 
\DeclarePairedDelimiter{\norm}{\lVert}{\rVert}

\DeclareMathOperator{\supp}{supp}

\newcommand{\R}{\mathbb{R}}
\newcommand{\C}{\mathbb{C}}

\newcommand{\N}{\mathbb{N}}
\newcommand{\normeq}[1]{{\left\vert\kern-0.25ex\left\vert\kern-0.25ex\left\vert #1 
    \right\vert\kern-0.25ex\right\vert\kern-0.25ex\right\vert}}

\usepackage{braket}
\usepackage{color} 
\newcommand{\Rn}{{\mathbb R^n}}

\newcommand{\alphaAppell}{a}
\newcommand{\betaAppell}{b}

\title{Mass Propagation for Electromagnetic Schr\"odinger Evolutions}

\author[1]{Juan Antonio Barcel\'o} 
\author[2]{Biagio Cassano}
\author[3]{Luca Fanelli}		

\affil[1]{ETSI de Caminos, Universidad Polit\'ecnica de Madrid, 28040 Madrid, Spain; juanantonio.barcelo@upm.es}
	
\affil[2]{Dipartimento di Matematica, Università degli Studi di Bari Aldo Moro, Via Edoardo Orabona 4, 70125 Bari, Italy; biagio.cassano@uniba.it}

\affil[3]{Ikerbasque \& Departamento de Matem\'aticas, Universidad del Pa\'is Vasco/Euskal Herriko Unibertsitatea (UPV/EHU), Barrio Sarriena s/n, 48940, Leioa, Spain; luca.fanelli@ehu.es}

\begin{document}

\date{\small \today}
%%%%%%%%%%%%%%%%%%%%%%%%%%%%%%%%%%%%%%%%%%%%%%%%%%%%%%%%%%%%%%%%%%%%%%%%%%%%%%%%%%%%%%%%%%%%%%%%%%%%%%%%%%%%%%%%%%%%%%%%%%%%%%%%%%%%%%%%%%

\maketitle
%\vspace{-1cm}

%\nocite{*}

%------------abstract------------%

% \begin{abstract}
%   \noindent
% 	This paper is devoted to providing quantitative bounds on the
%         location of eigenvalues, both \emph{discrete} and
%         \emph{embedded}, of non self-adjoint Lamé operators of
%         elasticity $-\Delta^\ast + V$ in terms of suitable norms of
%         the potential $V$. In particular, this allows to get sufficient conditions on the size of the potential such that the point spectrum of the perturbed operator remains empty. In three dimensions we show full spectral stability under suitable form-subordinated perturbations: we prove that the spectrum is purely continuous and coincides with the non negative semi-axis as in the free case.
% \end{abstract}

% ---------------------------------%

\begin{abstract}
  % Inspired by \cite{agirre2019some}, w
  We investigate the validity of gaussian lower bounds for solutions to an electromagnetic Schr\"odinger equation with a bounded time-dependent complex electric potential and a time-independent vector magnetic potential. We prove that, if a suitable geometric condition is satisfied by the vector potential, then positive masses inside of a bounded region at a single time propagate outside the region, provided a suitable average in space-time cylinders is taken.
\end{abstract}
%%%%%%%%%%%%%%%%%%%%%%
%%%%%%%%%%%%%%%%%%%%%%%

\section{Introduction}\label{sec.Introduction}
The study of an electromagnetic Schr\"odinger equation of the form
\begin{equation}\label{eq:main}
  \partial_t u = i (\Delta_A + V)u,
\end{equation}
where $u=u(x,t):\R^n \times  [0,1]\to\C$, and
\begin{align*} 
   &V=V(x,t)\colon \R^n\times [0,1] \to \C, \\ 
   &\Delta_A:= \nabla_A^2,\quad
  \nabla_A:=\nabla -i A,\quad A=A(x)\colon \R^n\to\Rn
\end{align*}
has been catching the interest  of the applied mathematics community for several years. The free Schr\"odinger equation (i.e. equation \eqref{eq:main} with $A\equiv V\equiv0$) is a somehow canonical example in PDE's, due to its deep connection with Fourier Analysis, which is clear by the solution formula for a datum $f\in L^2(\R^n)$
\begin{equation}\label{eq:solving}
  u(x,t):=e^{it\Delta}f(x) = (2\pi it)^{-\frac{n}{2}}e^{i\frac{|x|^2}{4t}}
  \mathcal F\left(e^{i\frac{|\cdot|^2}{4t}}f\right)\left(\frac{x}{2t}\right),
  \qquad
  \text{where }\mathcal Fg(x) :=\int_{\R^n}e^{-ix\cdot y}g(y)\,dy.
\end{equation}
A natural question, arising in Fourier Analysis, is concerned with the fastest possible simultaneous decay which a function $f$ and its Fourier transform $\hat f$ can enjoy, without being null. It is well known that Gaussians are the sharpest objects in this sense, as stated by the
 {\it Hardy's Uncertainty Principle}:

{\it if $f(x)=O\left(e^{-|x|^2/\beta^2}\right)$ and its Fourier transform
$\hat f(\xi)=O\left(e^{-4|\xi|^2/\alpha^2}\right)$, then}
\begin{align*}
  \alpha\beta<4
  &
  \Rightarrow
  f\equiv0
  \\
  \alpha\beta=4
  &
  \Rightarrow
  f\ is\ a \ constant\ multiple\ of\ e^{-\frac{|x|^2}{\beta^2}}.
\end{align*}
Formula \eqref{eq:solving} gives the intuition for the following PDE's version of the above Uncertainty Principle, which answers the analogous question about the fastest possible decay of a solutions of the Schr\"odinger Equation at two distinct times:

{\it if $u(x,0)=O\left(e^{-|x|^2/\beta^2}\right)$ and $u(x,T):=e^{iT\Delta}u(x,0)=
O\left(e^{-|x|^2/\alpha^2}\right)$, then}
\begin{align*}
  \alpha\beta<4T
  &
  \Rightarrow
  u\equiv0
  \\
  \alpha\beta=4T
  &
  \Rightarrow
  u(x,0)\ is\ a \ constant\ multiple\ of\ e^{-\left(\frac{1}{\beta^2}+\frac{i}{4T}\right)|x|^2}.
\end{align*}
The corresponding $L^2$-versions of the previous results were proved in \cite{CP} and affirm the following:
\begin{align*}
  e^{|x|^2/\beta^2}f \in L^2,\  e^{4|\xi|^2/\alpha^2}\hat f \in L^2,\
  \alpha\beta\leq4
  &
  \Rightarrow
  f\equiv0
  \\
  e^{|x|^2/\beta^2}u(x,0)\in L^2,\  e^{|x|^2/\alpha^2}e^{iT\Delta}u(x,0) \in L^2,\
  \alpha\beta\leq4T
  &
  \Rightarrow
  u\equiv0.
\end{align*}
We address the reader to \cite{BD, FS, SST, SS} as standard references about this topic.
If an electromagnetic field is present (i.e. $A$ and $V$ are not null in \eqref{eq:main}), it is usually difficult to exploit the Fourier Transform, in particular when the coefficients $A,V$ are rough. In the recent years,   
Escauriaza, Kenig, Ponce, and Vega
in the sequel of papers \cite{EKPV0,EKPV1,EKPV2,EKPV3,EKPV4, escauriaza2012uniqueness},
and  with Cowling in \cite{CEKPV} developed purely real analytical
methods to handle the above problems, which permits to obtain sharp
answers also in the case $V\neq 0$. Some analogous results have also
been obtained by the authors of the present manuscript in \cite{BFGRV,
  CF1, CF2}, in the presence of a non-trivial magnetic field.
The interested reader can also see \cite{FernV, FRRS, JLMP}, where analogous phenomena
are considered for discrete Schr\"{o}dinger evolutions. We refer to the recent survey \cite{FM} for more details and
references to further results.

\vskip0.3cm
\noindent
More recently, Agirre and Vega in \cite{agirre2019some}, inspired by the techniques of the above mentioned papers and motivated by the results in \cite{N}, answered to a similar question for the Schr\"odinger Equation, when the decay is assumed at only one time, instead of two. The main contribution in \cite{agirre2019some} is to prove the following: if a positive mass is present, for solutions of \eqref{eq:main} with $A\equiv0$ and $V$ bounded, inside of some region (a ball) at one time, then one also observes this mass outside the region, if a suitable time average is taken. This fact can be mathematically translated into a gaussian lower bound for solutions in suitable space-time cylinders. The arguments by Agirre and Vega are purely real analytical, and rely on suitable Carleman estimates.
\vskip0.3cm
\noindent
As we saw in \cite{BFGRV, CF1, CF2}, the presence of a magnetic field can produce interesting phenomena in this setting. In particular, since the fields lines are closed curves, the dispersive phenomenon and the space decay of solutions can be strongly influenced by a non-null magnetic potential, and a mathematical investigation about the relevant quantities related to $A$ has been performed, starting by \cite{fanelli2009non, FV}. In view of these considerations, the purpose of this note is to complete the results in \cite{agirre2019some}, by considering the more general model in which $A\neq0$. 
\vskip0.3cm
\noindent
Given $A=(A^1,\dots,A^n)\colon \Rn \to \Rn$ a real
vector field (\emph{magnetic potential}), we denote by $B : \Rn \to
M_{n\times n}(\R)$
the \emph{magnetic field}, namely the antisymmetric gradient of $A$,
given by
\begin{equation}\label{eq:Bdef}
  B(x)=D_x A(x)-D_x A^t(x),
  \qquad
  B_{jk}(x)=\partial_{x_j} A^k(x)- \partial_{x_k} A^j(x).
\end{equation} 
In dimension $n=3$, $B$ is identified with the vector field $\text{curl}\,A$, by the elementary properties of antisymmetric matrices.
We are now ready to state the last result of this paper.
%%%%%%%%%%%%%%%%%%
\begin{theorem}\label{thm:main}
Let $n\geq3$, $u\in\mathcal{C}([0,1];H^1_{loc}(\mathbb{R}^n))$ be a solution of \eqref{eq:main},
and assume that
\begin{equation}\label{eq:V.bound}
\norm{V}_{L^{\infty}(\mathbb{R}^n\times[0,1])} = M_V <+\infty,
\end{equation}
\begin{equation}\label{eq:condition.A}
  \int_0^1A(sx)\,ds\in\R^n,
  \qquad
  \text{for a.e. }x\in\R^n,
\end{equation}
\begin{equation}\label{eq:decay.B}
  \norm{x^t B}_{L^\infty(\R^n)} =: M_B < +\infty.
\end{equation}
Assume moreover that there exists a fixed vector $v \in \R^n$ such that 
  \begin{equation}  \label{eq:B.null}
  v^t B(x) = 0, \quad \text{ for a.e. }x \in \R^n.
\end{equation}
Finally assume that there exist $R_0,M_u>0$,  such that $R_1 > 4(R_0+1)$ and 
\begin{equation}\label{eq:initial.mass}
\int_{\abs{x}\leq R_0} \abs{u(x,0)}^2 dx = M_u^2,
\end{equation}
\begin{equation}\label{eq:norm.bound}
  \sup_{0\leq t \leq 1} \int_{\abs{x}\leq R_1}
  \left(|u(x,t)|^2+|\nabla_A u(x,t)|^2 \right) dx =: E_u^2 < +\infty.
\end{equation}
Then, there exist $t^\ast = t^\ast(R_0, R_1, M_V, M_B,  M_u, E_u)>0$
and $C = C(M_B) >0$ such that
\begin{equation}\label{eq:thesis}
  M_u^2 \leq C\frac{ e^{C \frac{\rho^2}{t}}}{t}\int_{t/4}^{3t}
  \int_{||y|-\rho-\rho\frac{s}{t}|<4(\rho+1)\sqrt{t}}
  \left(\vert u(y,s)\vert^2 +
  s\vert \nabla_{A} u(y,s) \vert^2\right)\,dyds, \ \ \ \ \  R_0 \leq \rho \leq R_1,
\end{equation}
for any $t\in(0,t^\ast)$.
\end{theorem}
%%%%%%%%%%%%%%%%%%
\begin{remark}\label{rem:1}
Notice that the assumptions of Theorem \ref{thm:main} are gauge invariant. As for condition \eqref{eq:condition.A}, it has to be understood as a necessary local integrability condition, in order to get the freedom to choose the so called {\it Cr\"onstrom gauge} (see Section \ref{3sec:cronstrom} below). In particular, condition \eqref{eq:condition.A} is not satisfied in the case of homogeneous vector potentials $A$ of degree $-1$, which is the case when the Hamiltonian $\Delta_A$ is scaling invariant. A well known example is given by the {\it Aharonov-Bohm}-type potential $A(x) = \lambda(0,\dots,0,-x_{n},x_{n-1})/(x_{n-1}^2+x_n^2)\in\R^n$, for which the validity of Theorem \ref{thm:main} is still an open question.
\end{remark}
\begin{remark}\label{rem:2}
  The choice of the time interval $[0,1]$ does
  not lead the generality of the results. Indeed, $v\in
  C([0,T],L^2(\Rn))$ 
  is solution to \eqref{eq:main} in \mbox{$\Rn\times[0,T]$}
  if and only if $u\colon \R^n \times [0,1]\to \C$, $u(x,t)=T^{\frac{n}{4}}v(\sqrt{T}x,Tt)$ is solution to
\begin{equation*}
  \partial_t u = i(\Delta_{A_T}u + V_T (x,t) u)
  \quad \text{ in }\Rn \times [0,1],
\end{equation*} 
where
\begin{equation*} 
A_T(x)=\sqrt{T}A(\sqrt{T}x),\quad V_T(x)=T
  V(\sqrt{T}x,Tt).
\end{equation*} 
\end{remark} 
\begin{remark}\label{rem:3}
Due to assumption \eqref{eq:B.null}, which is crucial in the proof of
the main theorem, the case $n=2$ is not included in the
result. Indeed, there does not exist any non-null anti-symmetric
$2\times2$-matrix with a fixed null vector. For explicit examples in
dimension $n\geq3$, we refer to \cite{BFGRV}. This leaves an
interesting open question about the validity of Theorem \ref{thm:main}
in 2D, were our arguments do not work; we conjecture it should be possible to produce explicit counterexamples to the result.
\end{remark}

We complement Theorem \ref{thm:main} with the following results of
uniqueness for solutions to \eqref{eq:main}. They are immediate
consequences of Theorem \ref{thm:main} so their proof will be
omitted. 
\begin{corollary}
  In the assumptions of Theorem \eqref{thm:main}, let $u\in\mathcal{C}([0,1];H^1(\mathbb{R}^n))$ be a
  solution of \eqref{eq:main}.
  \begin{itemize}
  \item If there exist $(R_j)_{j \in \N}$, $R_j \to +\infty$ such
      that  for all $j \in \N$
\begin{equation*}
\lim_{t \to 0} C\frac{e^{C \frac{R_j ^2}{t}}}{t} \int_{t/4}^{3t}\int_{||y| - R_j(1+s/t)| < 4(R_j+1)\sqrt{t}} 
\left(|u(y,s)|^2+s\left|\nabla u(y,s)\right|^2\right) dy ds=0,
\end{equation*}
then $u\equiv 0$;
\item  if there
  exists $(t_j)_{j \in \N}\subset(0,t^*)$, $t_j \to 0$ such that for all $j
  \in \N$
  \begin{equation*}
    \lim_{\rho \to +\infty}
    C \frac{e^{C \frac{\rho^2}{t_j}}}{t_j} \int_{t_j/4}^{3t_j}\int_{||y| - \rho(1+s/t_j)| < 4(\rho+1)\sqrt{t_j}} 
  \left(|u(y,s)|^2+s\left|\nabla u(y,s)\right|^2 \right) dy ds=0,
\end{equation*}
then $u\equiv 0$.
\end{itemize}
\end{corollary}

The rest of the paper is devoted to the proof of Theorem \ref{thm:main}. 

\subsection*{Acknowledgments}
Juan Antonio Barcel\'o acknowledges the support of Ministerio de Ciencia, Innovaci\'on y Universidades of the Spanish goverment through grant MTM2017-85934-C3-3-P.
Biagio Cassano is member of GNAMPA (INDAM) and he is supported by Fondo Sociale
Europeo – Programma Operativo Nazionale Ricerca e Innovazione
2014-2020, progetto PON: progetto AIM1892920-attivit\`a 2, linea 2.1. 
\section{Preliminaries}
We start with a preliminary section, in which we show the fundamental
tools for the proof of our main theorem.

It is useful to generalise \eqref{eq:Bdef} to consider a general
time-dependent magnetic potential $A=(A^1,\dots,A^n)\colon \Rn \times
[0,1] \to \Rn$: the magnetic field $B : \Rn \times
[0,1]\to M_{n\times n}(\R)$
is then given by
\begin{equation}\label{eq:Bdef}
  B(x,t)=D_x A(x,t)-D_x A^t(x,t),
  \qquad
  B_{jk}(x,t)=\partial_{x_j} A^k(x,t)- \partial_{x_k} A^j(x,t).
\end{equation} 

\subsection{Cr\"onstrom gauge}\label{3sec:cronstrom}
Equation \eqref{eq:main} is invariant under gauge transformations, namely
 if $u$ solves \eqref{eq:main}, then, given $\varphi=\varphi(x):\R^n\to\R$,
the function $\widetilde u=e^{-i\varphi}u$ is a solution to
\begin{equation*}
  \partial_t\widetilde u= i\left(\Delta_{\widetilde A}\widetilde u+V(x,t)\widetilde u\right),
\end{equation*}
where 
$\widetilde
A=A - \nabla\varphi$. This invariance is useful since it allows us to
fix the most appropriate gauge, to simplify the computations. As in
\cite{BFGRV}, we use here the {\it Cr\"onstrom gauge} (also called
{\it transversal} or {\it Poincar\'e} gauge), which is given by the following condition
\begin{equation}\label{3eq:cronstrom3}
  x\cdot \widetilde A(x)=0 \quad \text{ for a.e. }x \in  \R^n.
\end{equation} 
If $A$ satisfies \eqref{eq:condition.A}, it is always possible, via gauge transformation, to reduce to the case in which \eqref{3eq:cronstrom3} holds, as the following classical result by \cite{I} shows.
\begin{lemma}[\cite{I}]
  \label{3lem:cronstrom1}
  Let $n\geq2$, $A=A(x)=(A^1,\dots,A^n):\R^n\to\R^n$, 
 $B := DA-DA^t$, and $\Psi(x):=x^tB(x)\in\R^n$ for all $x
 \in \R^n$.
Assume that 
\begin{equation}\label{eq:numero}
  \int_0^1\Psi(sx)\,ds\in\R^n,
  \qquad
  \int_0^1A(sx)\,ds\in\R^n,
  \qquad
  \text{for a.e. }x\in\R^n
  \end{equation}
  and denote by
  \begin{gather}
    \label{3eq:cronstrom2}
    \widetilde A(x):= -\int_0^1\Psi(sx)\,ds,
    \\ \label{3eq:varphi}
    \varphi(x):=x\cdot\int_0^1A(sx)\,ds\in\R.
  \end{gather}
  Then $B = D\widetilde A - D \widetilde A^t$,
  $\widetilde A = A -\nabla \varphi$
    % \quad
    % x^t D\widetilde A(x)  = -\Psi(x) +\int_0^1\Psi(sx)\,ds,
  and \eqref{3eq:cronstrom3} holds true.
\end{lemma}
See \cite[Lemma 2.2]{BFGRV} for the details.
\subsection{Appell Transformation}\label{3sec:Appell}
We generate a family of solutions to \eqref{eq:main} by means of the
following pseudoconformal transformation (Appell).
\begin{lemma}[\cite{BFGRV}, Lemma 2.7]\label{lem:appell}
  Let $A =(A^1(y,s),\dots,A^n(y,s)):\Rn \times [0,1] \to \R^{n}$,
\mbox{$V=V(y,s)$}, $F=F(y,s):\R^n \times [0,1] \to\C$, $u=u(y,s):\R^{n}\times[0,1]\to\C$ be a solution to
  \begin{equation}\label{3eq:1appell}
  \partial_s u=i\left(\Delta_Au+V(y,s)u+F(y,s)\right) \quad \text{ in  }
  \Rn \times [0,1]
\end{equation}
and define, for any $\alphaAppell,\betaAppell>0$, the function
\begin{equation}\label{3eq:appell}
  \widetilde u(x,t):=
  \left(\frac{\sqrt{\alphaAppell\betaAppell}}{\alphaAppell(1-t)+\betaAppell t}\right)^{\frac
    n2} \,
  u\left(\frac{x\sqrt{\alphaAppell\betaAppell}}{\alphaAppell(1-t)+\betaAppell
      t},\frac{t\betaAppell}{\alphaAppell(1-t)+\betaAppell t}\right)
  e^{\frac{(\alphaAppell-\betaAppell)|x|^2}{4i (\alphaAppell(1-t)+\betaAppell t)}}.
\end{equation}
Then $\widetilde u$ is a solution to
\begin{equation}\label{3eq:2appell2}
  \partial_t\widetilde u=i\left(\Delta_{\widetilde A}\widetilde u
    +\frac{(\alphaAppell-\betaAppell)\widetilde A\cdot x}
    {(\alphaAppell(1-t)+\betaAppell t)}\widetilde u+
    \widetilde V(x,t)\widetilde u+\widetilde F(x,t)\right)
  \quad \text{ in } \Rn \times [0,1],
\end{equation}
where
\begin{align}
  \label{3eq:Aappell}
  \widetilde A(x,t)
  &
  = \frac{\sqrt{\alphaAppell\betaAppell}}{\alphaAppell(1-t)+\betaAppell t}
  A\left(\frac{x\sqrt{\alphaAppell\betaAppell}}{\alphaAppell(1-t)+\betaAppell t},\frac{t\betaAppell}{\alphaAppell(1-t)+\betaAppell t}\right)
  \\
  \label{3eq:Vappell}
  \widetilde V(x,t)
  &
  = \frac{\alphaAppell\betaAppell}{(\alphaAppell(1-t)+\betaAppell t)^2}
  V\left(\frac{x\sqrt{\alphaAppell\betaAppell}}{\alphaAppell(1-t)+\betaAppell t},\frac{t\betaAppell}{\alphaAppell(1-t)+\betaAppell t}\right)
  \\
  \label{3eq:Fappell}
  \widetilde F(x,t)
  &
  = \left(\frac{\sqrt{\alphaAppell\betaAppell}}{\alphaAppell(1-t)+\betaAppell t}\right)^{\frac n2+2}
  F\left(\frac{x\sqrt{\alphaAppell\betaAppell}}{\alphaAppell(1-t)+\betaAppell t},\frac{t\betaAppell}{\alphaAppell(1-t)+\betaAppell t}\right)
  e^{\frac{(\alphaAppell-\betaAppell)|x|^2}{4i
  (\alphaAppell(1-t)+\betaAppell t)}}.
\end{align}
\end{lemma}

\subsection{Carleman estimate}
We now show the main tool, which is a suitable Carleman estimate for the purely magnetic Schr\"odinger group $i\partial_t+\Delta_A$.
Here we adapt \cite[Lemma
3.1]{EKPV0} to allow the presence of a magnetic
potential.
\begin{lemma}\label{lem:carleman}
  Let $n\geq3$, $R>1$ and $\varphi : [0,1] \to \R$ a smooth function.
  Let $A=A(x,t):\Rn \times [0,1] \to\R^n$,
  $B:=D_x A-D_x A^t:\Rn \times [0,1] \to M_{n\times n}(\R)$,  and assume that
  there exists a fixed vector $v \in \R^n$ such that 
\begin{gather}\label{eq:cronstrocond}
  x\cdot A_t(x,t) = 0, \quad v\cdot A_t(x,t) = 0,
  \quad v^t B(x,t) = 0 \quad
  \text{ for a.e. }x\in\R^{n}, t \in [0,1],\\
  \label{eq:decay.B.carleman}
  \norm{x^t B}_{L^\infty(\R^{n}\times [0,1])}   < +\infty.
\end{gather}
  Then,
  there exists $c = c(\norm{\varphi'}_{\infty},\norm{\varphi''}_{\infty},   \norm{x^t B}_{\infty})>0$ such that
  \begin{equation}\label{eq:Carleman}
    \frac{\tau^{3/2}}{c R^2}
    \norm*{e^{\tau\abs{\frac{x}{R}+\varphi(t) v }^2}g}_{L^2(\R^{n}\times[0,1])}
    \leq
    \norm*{e^{\tau\abs{\frac{x}{R}+\varphi(t) v }^2} (i \partial_t +
      \Delta_A)g}_{L^2(\R^{n}\times[0,1])}
  \end{equation}
  for all $\tau\geq c R^2$ and for all $g \in
  C_c^{\infty}(\Rn\times [0,1])$ with
  \begin{equation}\label{eq:supp}
    \supp g \subset \left\{ (x,t) \in \R^{n}\times [0,1] :
      \left\vert \frac{x}{R}+\varphi(t) v \right\vert \geq 1  \right\}.
  \end{equation}
\end{lemma}
\begin{remark}
In the proof of Theorem \ref{thm:main}, after using the Appell Transformation, we are reduced to an equation with a time-dependent magnetic potential. This motivates the necessity to prove the Carleman estimate \eqref{eq:Carleman} for $A=A(t,x)$ satisfying the conditions in \eqref{eq:cronstrocond}.
\end{remark}
\begin{proof}[Proof of Lemma \ref{lem:carleman}]
  Without loss of generality, we may assume \eqref{eq:cronstrocond} with $v = e_1
  = (1,0,\dots,0) \in \R^n$.
Denoting by $f = e^{\tau\abs{\frac{x}{R}+\varphi(t)e_1}^2}g$,
an explicit computation shows that
\begin{equation*}
  e^{\tau\abs{\frac{x}{R}+\varphi(t)e_1}^2}
  (i \partial_t + \Delta_A) g
  =
  S_\tau f - 4 \tau A_\tau f, 
\end{equation*}
where $S_\tau$ and $A_\tau$ are respectively the symmetric and
anti-symmetric operators
\begin{equation*}
  \begin{split}
    & S_\tau = i \partial_t + \Delta_A +
    \frac{4\tau^2}{R^2} \left\vert \frac{x}{R} + \varphi
      e_1\right\vert^2,
    \\
    & A_\tau = \frac{1}{R} \left( \frac{x}{R} + \varphi e_1 \right)
    \cdot \nabla_A + \frac{n}{2R^2} + \frac{i\varphi'}{2}
    \left(\frac{x_1}{R} + \varphi \right).
  \end{split}
\end{equation*}
We hence have
\begin{equation*}
  \begin{split}
  \norm{e^{\tau\left\vert \frac{x}{R} + \varphi e_1 \right\vert^2}
      (i\partial_t +\Delta) g}_{L^2(\R^{n} \times [0,1])}^2
    & \geq 
    -4\tau
    \langle [S_\tau,A_\tau] f ,f \rangle_{L^2(\R^{n} \times [0,1])}.
  \end{split}
\end{equation*}
An explicit computation (see \cite[Lemma 4.1]{BFGRV}) shows that
\begin{equation}\label{eq:use.me}
  \begin{split}
  &\big\Vert{e^{\tau\left\vert  \frac{x}{R} + \varphi e_1 \right\vert^2}
    (i\partial_t +\Delta) g\big\Vert}_{L^2(\R^{n} \times [0,1])}^2
  \\
  & \geq 
    \frac{32 \tau^3}{R^4} \int_0^1 \int_{\R^{n}}  \left\vert \frac{x}{R} + \varphi e_1
    \right\vert^2
    \abs{f}^2 \,dx dt
    +
    \frac{8\tau}{R^2} \int_0^1 \int_{\R^{n}}  \abs{\nabla_A f}^2 \, dxdt
  \\
    & \quad + 
    2\tau\int_0^1 \int_{\R^{n}}  \left[ \left( \frac{x_1}{R} + \varphi
      \right) \varphi''
      + \varphi'^2 \right] \abs{f}^2 \, dxdt
    +\frac{8\tau}{R} \Im \int_0^1 \int_{\R^{n}}  \varphi' (e_1 \cdot \nabla_A) f
    \overline{f} \, dxdt
  \\
    & \quad -
    \frac{8\tau}{R} \Im \int_0^1 \int_{\R^{n}}  f \left(\frac{x}{R} + \varphi e_1
    \right)^t B \overline{\nabla_A f} \,dxdt
    -\frac{4\tau}{R}\Im \int_0^1 \int_{\R^{n}} 
    \left(\frac{x}{R}+\varphi e_1\right)\cdot A_t \abs{f}^2 \,dxdt.
  \end{split}
\end{equation}
By \eqref{eq:cronstrocond}, the last term at right hand side
vanishes. Thanks to \eqref{eq:cronstrocond} and
\eqref{eq:decay.B.carleman}
we estimate
\begin{equation}\label{eq:use.me.1}
  \begin{split}
  & - \frac{8\tau}{R} \Im \int_0^1 \int_{\R^{n}}  f \left(\frac{x}{R} + \varphi e_1
  \right)^t B \overline{\nabla_A f} \,dxdt
  =
   - \frac{8\tau}{R^2} \Im \int_0^1 \int_{\R^{n}}  f  {x^t}
   B \overline{\nabla_A f} \,dxdt
  \\
  &\geq - \frac{4\tau}{R^2}\norm{x^t B}_{L^\infty(\R^{n+1})}^2
  \int_0^1 \int_{\R^{n}}  \abs{f}^2\,dxdt
  -\frac{4\tau}{R^2}\int_0^1 \int_{\R^{n}}  \abs{\nabla_A f}^2 \,dxdt.
\end{split}
\end{equation}
Also, we have
\begin{equation}\label{eq:use.me.2}
  \frac{8\tau}{R} \Im \int_0^1 \int_{\R^{n}}  \varphi' (e_1 \cdot \nabla_A) f
  \overline{f} \, dxdt
  \geq
  -{4\tau}\norm{\varphi'}_{L^\infty([0,1])}^2
  \int_0^1 \int_{\R^{n}} \abs{f}^2 \,dxdt
  -\frac{4\tau}{R^2}   \int_0^1 \int_{\R^{n}} \abs{\nabla_A f}^2 \,dxdt.
\end{equation}
From \eqref{eq:use.me}, since $R>1$ and thanks to \eqref{eq:supp}, \eqref{eq:use.me.1} and
\eqref{eq:use.me.2}, we have
\begin{equation*}
  \begin{split}
    &\big\Vert{e^{\tau\left\vert \frac{x}{R} + \varphi e_1 \right\vert^2}
      (i\partial_t +\Delta) g\big\Vert}_{L^2(\R^{n} \times [0,1])}^2 
    \\ &\geq 
    \left[    \frac{32 \tau^3}{R^4} - 2
   \tau \left(\norm{\varphi''}_{L^\infty([0,1])} +
   \norm{\varphi'}_{L^\infty([0,1])}^2
   +2  \norm{x^t B}_{\infty}^2\right) \right]
   \int_0^1 \int_{\R^{n}}  \left\vert \frac{x}{R} + \varphi e_1\right\vert
    \abs{f}^2 \,dx dt.
  \end{split}
\end{equation*}
The coefficient of the first term at right hand side is bigger than
$\sigma^3/c^2 R^4$, if $\sigma \geq c R^2$ for a suitable $c =
c(\norm{\varphi'}_{\infty}, \norm{\varphi''}_{\infty}, \norm{x^t
  B}_{\infty})>0$, therefore \eqref{eq:Carleman} follows thanks to \eqref{eq:supp}.
\end{proof}

\section{Proof of \Cref{thm:main}}
 The proof is an adaptation of the proof of \cite[Theorem 2.1]{agirre2019some}, taking into
  account the presence of a magnetic potential.
  In the following, without loss of generality, we assume that $v = e_1 = (1,0,\dots,0)
  \in \R^n$.
  \vskip0.3cm
  \noindent
\underline{\bf Reduction to the Cr\"onstrom gauge}. We start by a gauge transformation.
By Lemma \ref{3lem:cronstrom1}, thanks to assumption
\eqref{eq:condition.A} and denoting by
  \begin{equation*}
    \varphi(x):= x \cdot\int_0^1 A(sx)\,ds,
    \quad
    \widetilde A(x):=A(x) - \nabla \varphi(x)
\end{equation*}
we have that $B =DA-DA^t = D\widetilde A - D \widetilde
  A^t$, and for a.e. $x \in \R^n$ 
  \begin{gather}
    \label{eq:cronstrom.1}
    \widetilde A(x)= -\int_0^1 (sx)^t B(sx) \,ds,
\\
    x \cdot   \widetilde A(x) = 0.
  \end{gather}
Moreover, from \eqref{eq:B.null} and \eqref{eq:cronstrom.1} we see that
  \begin{equation}\label{eq:ventisette}
    e_1 \cdot \widetilde A(x)  =0, \quad \text{ for a.e. }x \in \R^n.
  \end{equation}
Let $\widetilde u := e^{-i\varphi } u$. Then
(cfr.~\Cref{3sec:cronstrom}) $\widetilde u$ is solution to 
\begin{equation}\label{eq:nuovanuova}
    \partial_t\widetilde u= i\left(\Delta_{\widetilde A}\widetilde u+
      V(x,t) \widetilde u\right)
    \quad \text{ in }\Rn \times [0,1]
\end{equation}
and the conditions in \eqref{eq:initial.mass} and \eqref{eq:norm.bound}
are true replacing $u,A$ with $\widetilde u, \widetilde A$.
\vskip0.3cm
\noindent
\underline{{\bf Appell Transformation}}. We now apply Lemma \ref{lem:appell} to solutions of \eqref{eq:nuovanuova}. To lighten the notations, in the following we will omit the tildes and just denote
$\widetilde u$, $\widetilde A$
by $u$ and $A$ in \eqref{eq:nuovanuova}.
We choose
  \begin{equation*}
    \alphaAppell, \betaAppell >0, \quad \gamma := \frac{\alphaAppell}{\betaAppell},
  \end{equation*}
  in such a way that
    \begin{equation}
    \label{eq:gamma.big}
    \gamma > \gamma^\ast := \text{max}\left(1, \frac{2}{R_0},
      \frac{ 64  E_u^2 (1+M_V)}{M_u^2}, \frac{4}{R_1 - 4R_0},
      \sqrt{\frac{M_V M_u}{2^{12} E_u}},
      \frac{2^{8} E_u}{\sqrt{c} R_0 M_u}
    \right)^2,
  \end{equation}

where $c$ is defined later in \eqref{eq:carleman.proof}. Let
\begin{equation}\label{eq:defn.v}
v(x,t) := \alpha(t)^{\frac{n}2}e^{-\frac{i}{4}\beta(t)\vert x\vert^2}u(\alpha(t)x,s(t)), \quad (x,t)\in\mathbb{R}^n\times [0,1],
\end{equation}
with
\begin{equation*}
    \alpha(t)= \frac{1}{(1-t)\sqrt{\gamma}+t/\sqrt{\gamma}}, \quad
    \beta(t)=\frac{1}{1-t+t/\gamma}-\frac{1}{\gamma(1-t)+t},
    \quad
    s(t)=\frac{t}{\gamma(1-t)+t}.
  \end{equation*}
    Thanks to \Cref{lem:appell}, $v$ is solution to 
    \begin{equation}\label{3eq:2appell}
      \partial_t v 
      = i\left(\Delta_{\widetilde A} v
        +  \widetilde V(x,t) v \right)
      \quad \text{ in }\Rn \times [0,1]
    \end{equation}
for $\widetilde A$ and $\widetilde V$  defined by % as in
% \eqref{3eq:Aappell} and \eqref{3eq:Vappell} respectively:
\begin{equation}
  \label{eq:A.V.appell}
    \widetilde A(x,t)
  := \alpha(t) \, A(\alpha(t)x),
  \quad 
  \widetilde V(x,t)
  := (\alpha(t))^2 \,  V(\alpha(t) x , s(t)).
\end{equation}
We remark
that, since $A$ is in the Cr\"onstrom gauge, then $\widetilde A$
is in the Cr\"onstrom gauge too. Also, we have that
$\norm{\widetilde V}_{L^{\infty}(\R^n\times[0,1])} = \gamma M_V < +\infty$. In addition, by \eqref{eq:ventisette} and \eqref{eq:A.V.appell}, we see that
 \begin{equation}\label{eq:ventisettedue}
    e_1 \cdot \widetilde A(x,t)  = e_1\cdot \widetilde A_t(x,t)=0,
    \quad \text{ for a.e. }(x,t) \in \R^n\times [0,1].
  \end{equation}
  \underline{{\bf Conclusion of the proof}}.
Let us denote
\begin{equation}\label{eq:erre}
R := R_0 \sqrt{\gamma},
\end{equation}
so that from \eqref{eq:gamma.big} we
  have $R > 2$. We define the following auxiliary functions:
  \begin{equation*} 
    \theta_R, \eta\in\mathcal{C}^{\infty}(\mathbb{R}^n), \quad
    \theta_R(x)= 
    \begin{cases}
      1 & \text{ if }|x|\leq R \\
      0 & \text{ if }|x|\geq R+1,
    \end{cases}
  \quad
 \eta(x)= 
 \begin{cases}
   1 & \text{ if }|x|\geq 2 \\
   0 & \text{ if }|x|\leq 3/2,
 \end{cases} 
\end{equation*}
 such that for all $x\in \Rn$
\begin{align}
  \label{eq:control.derivatives.theta}
  & \abs{\theta_R}\leq 1, \quad \abs{\nabla\theta_R(x)} \leq 1,  \quad \abs{\Delta \theta_R(x)}
    \leq 2, \\
    \label{eq:control.derivatives.eta}
  & \abs{\eta(x)}\leq 1, \quad \abs{\nabla\eta(x)} \leq 2,  \quad \abs{\Delta\eta(x)} \leq
    4.
\end{align}
Moreover, let
\begin{equation*}
 \varphi\in\mathcal{C}^{\infty}([0,1]), \quad \varphi(t)= 
 \begin{cases}
   4 & \text{ if }t\in [3/8,5/8] \\
   0 & \text{ if }t\in [0,1/4]\cup [3/4,1],
 \end{cases}
\end{equation*}
 such that for all $t \in [0,1]$
\begin{equation}
  \label{eq:control.derivatives.varphi}
  \abs{\varphi(t)}\leq 4, \quad  \abs{\varphi'(t)}\leq 32.
\end{equation}
We let
\begin{equation}\label{eq:defn.g}
g(x,t) = \theta_R(x) \, \eta\left(\frac{x}{R}+\varphi(t)e_1\right)\, v(x,t),
\quad (x,t)\in\mathbb{R}^n\times [0,1].
\end{equation}
We observe that $\supp g$ is compact and
\begin{equation}\label{eq:supp.g}
  \supp g \subset \left\{(x,t) \in \R^{d}\times[0,1] \, \middle| \,
    \abs{x}\leq R+1, \,
    \frac32 \leq \left|\frac{x}{R}+\varphi(t) e_1 \right|, \,
    t \in \left[\frac14,\frac34\right]
    \right\},
\end{equation}
indeed for $t \in [0,\tfrac14] \cup [\tfrac34,1]$, $g(x,t)$ is non vanishing
if $\frac{3}{2} \leq \frac{\abs{x}}{R}
\leq \frac{R+1}{R}$, that is in contraddiction with $R > 2$ given by \eqref{eq:gamma.big}.
Thanks to \Cref{lem:carleman} which can be applied under our assumptions, there exists
$c=c(\norm{\varphi'}_{\infty},\norm{\varphi''}_{\infty}, M_B)>0$ such that for all
$\tau \geq c R^2$
\begin{equation}\label{eq:carleman.proof}
  \frac{\tau^{3/2}}{cR^2}
  \norm*{e^{\tau\left| \frac{x}{R}+\varphi(t)e_1\right|^2}
    g(x,t)}_{L^2(\R^n \times [0,1])} \leq
  \norm*{ e^{\tau\left| \frac{x}{R}+\varphi(t)e_1\right|^2}
    (i\partial_t + \Delta_{\widetilde{A}})g(x,t)}_{L^2(\R^n \times [0,1])}.
\end{equation}
In the following we estimate from above and from below the
quantities in \eqref{eq:carleman.proof}.

We estimate from below the left hand side of
\eqref{eq:carleman.proof}: since 
$\left| \frac{x}{R}
  + \varphi(t)e_1\right| \geq 2$ and $g = \theta_R v$ on $\{\vert
x\vert\leq R+1\}\times [3/8,5/8]$,
we have
\begin{equation}\label{eq:from.below}
  \begin{split}
\norm*{e^{\tau\left| \frac{x}{R}+\varphi(t)e_1\right|^2}
  g(x,t)}_{L^2(\R^n \times [0,1])}^2
  & = \int_0^1\int_{\R^n} e^{2\tau\left| \frac{x}{R}+\varphi(t)e_1\right|^2}\vert g(x,t)\vert^2dxdt  \\
  & \geq e^{8\tau}\int_{3/8}^{5/8}\int_{\vert x\vert\leq R+1} \vert \theta_R(x)v(x,t)\vert^2dxdt \\
  & = e^{8\tau}\int_{3/8}^{5/8}\int_{\vert x\vert\leq R+1}
  \alpha(t)^n\vert \theta_R(x)u(\alpha(t)x,s(t))\vert^2dxdt.
\end{split}
\end{equation}
%  & 
It is convenient to perform the following change of variables in the
integral at right hand side of \eqref{eq:from.below}:
\begin{equation}\label{eq:change.of.variables}
y  = \alpha(t) x, \quad s(t)=\frac{t}{\gamma(1-t)+t}.
\end{equation}
% Observing that
% \begin{equation*}
% s\left(\left[\frac38,\frac58\right]\right)\subset
% \left[\frac{3}{5\gamma+3}, \frac{5}{3\gamma+5}\right], =:I,
% \end{equation*}
% with a simple computation one sees that
% \begin{equation*}
%   \frac{\gamma}{8}\leq \frac{dt}{ds}(s) \leq \gamma, 
% \end{equation*}
It is useful to observe that % De hecho es $\alpha(t) \leq \frac{8}{3\sqrt{\gamma}}$
\begin{align}\label{eq:for.change.of.variables.1}
  & \frac{1}{\sqrt{\gamma}} \leq \alpha(t) \leq \frac{4}{\sqrt{\gamma}},
  \quad \text{ for all } t \in \left[\frac38,\frac58\right],
  \\
  \label{eq:for.change.of.variables.2}
  & \frac{\gamma}{8}\leq \frac{dt}{ds}(s) =\frac{\gamma}{(1+s\gamma-s)^2} \leq \gamma,
  \quad \text{ for all }s\in 
  \left[\frac{3}{5\gamma+3}, \frac{5}{3\gamma+5}\right]
  = s\left(\left[\frac38,\frac58\right]\right).
\end{align}
From \eqref{eq:from.below} and \eqref{eq:for.change.of.variables.2}  we conclude that 
\begin{equation*}
  \norm*{e^{\tau\left| \frac{x}{R}+\varphi(t)e_1\right|^2}
  g(x,t)}_{L^2(\R^n \times [0,1])}^2
\geq e^{8\tau}\frac{\gamma}{8}\int_{\frac{3}{5\gamma+3}}^{\frac{5}{3\gamma+5}}
\int_{\vert y \vert \leq \alpha(t(s))(R+1)}
\left| \theta_R\left(\frac{y}{\alpha(t(s))}\right)\, u(y,s) \right|^2dyds.
\end{equation*}
Clearly then
\begin{equation}\label{eq:carleman.from.below}
  \norm*{e^{\tau\left| \frac{x}{R}+\varphi(t)e_1\right|^2}
  g(x,t)}_{L^2(\R^n \times [0,1])}^2
\geq e^{8\tau}\frac{\gamma}{8}\int_{\frac{3}{5\gamma+3}}^{\frac{5}{3\gamma+5}}
\int_{\vert y \vert \leq \alpha(t(s))(R+1)}
\left| \theta_R\left(\frac{y}{\alpha(t(s))}\right)\, u(y,0)
\right|^2dyds
+ e^{8\tau}\frac{\gamma}{8} E,
\end{equation}
with
\begin{equation*}
  E = \int_{\frac{3}{5\gamma+3}}^{\frac{5}{3\gamma+5}}
\int_{\vert y \vert \leq \alpha(t(s))(R+1)}
\theta_R^2\left(\frac{y}{\alpha(t(s))}\right)\, \left(\abs{u(y,s)}^2 - \abs{u(y,0)}^2\right) dyds.
\end{equation*}
We estimate from below the first term at right hand side in
\eqref{eq:carleman.from.below}: thanks to \eqref{eq:gamma.big} we have 
$\abs{  [\frac{3}{5\gamma+3}, \frac{5}{3\gamma+5}]
} > 1/(4\gamma)$ and 
thanks to \eqref{eq:initial.mass} and
\eqref{eq:for.change.of.variables.1} we conclude 
\begin{equation}\label{eq:to.conclude.1}
  \begin{split}
& e^{8\tau}\frac{\gamma}{8}\int_{\frac{3}{5\gamma+3}}^{\frac{5}{3\gamma+5}}
\int_{\vert y \vert \leq \alpha(t(s))(R+1)}
\left| \theta_R\left(\frac{y}{\alpha(t(s))}\right)\, u(y,0)
\right|^2dyds
\\ &  \geq
  e^{8\tau}\frac{\gamma}{8}\int_{\frac{3}{5\gamma+3}}^{\frac{5}{3\gamma+5}}
\int_{\vert y \vert \leq \alpha(t(s)) R}
\abs{u(y,0)}^2dyds
 \geq   \frac{e^{8\tau}}{32} \int_{\vert y \vert \leq  R /\sqrt{\gamma}}
\abs{u(y,0)}^2dy
\\ & = \frac{e^{8\tau}}{32}M_u^2.
\end{split}
\end{equation}
To estimate $E$, we observe that from \eqref{eq:main} we get that 
\begin{equation*}
  \frac{d}{dt}|u|^2  = -2\Im (\text{div}(u\overline{\nabla_{A} u}) + V|u|^2),
\end{equation*}
that gives
\begin{equation*}
|u(y,s)|^2-|u(y,0)|^2 = -2 \Im \int_0^s \left(\text{div}\left(u(y,s')\cdot
  \overline{\nabla_A u(y,s')}\right) + V(y,s') |u(y,s')|^2\right)ds'.
\end{equation*}
So
\begin{equation}\label{eq:E}
  \begin{split}
  \abs{E}
  \leq &2 \left\vert \int_{\frac{3}{5\gamma+3}}^{\frac{5}{3\gamma+5}} \int_0^s
    \int_{\vert y \vert \leq \alpha(t(s))(R+1)}
    \theta_R^2\left(\frac{y}{\alpha(t(s))}\right)\, \text{div} \left(u(y,s')\cdot
  \overline{\nabla_A u(y,s')}\right) dyds'ds
   \right\vert
   \\
   &+2 \left\vert \int_{\frac{3}{5\gamma+3}}^{\frac{5}{3\gamma+5}}
     \int_0^s \int_{\vert y \vert \leq \alpha(t(s))(R+1)}
    \theta_R^2\left(\frac{y}{\alpha(t(s))}\right)\,  V(y,s') |u(y,s')|^2 dyds'ds \right\vert 
   = E_1 + E_2.
 \end{split} 
\end{equation}
To estimate $E_1$ we integrate by parts: there is no boundary
contribution thanks to the choice of $\theta_R$.
Thanks to \eqref{eq:gamma.big}, \eqref{eq:control.derivatives.theta}
and 
\eqref{eq:for.change.of.variables.1} we have
% {\color{blue} quiero decir que
%   $(\theta_R\nabla\theta_R)\left(\frac{y}{\alpha(t(s))}\right)$ es la
%   funci\'on $(\theta_R\nabla\theta_R)$ computada en la variable
%   $\left(\frac{y}{\alpha(t(s))}\right)$}
\begin{equation}\label{eq:E1}
  \begin{split}
  E_1 & \leq
  4 \left\vert \int_{\frac{3}{5\gamma+3}}^{\frac{5}{3\gamma+5}}
   \frac1{\alpha(t(s))}
    \int_0^s
    \int_{\vert y \vert \leq \alpha(t(s))(R+1)}
    (\theta_R\nabla\theta_R)\left(\frac{y}{\alpha(t(s))}\right)\, u(y,s')\cdot
    \overline{\nabla_A u(y,s')} dyds'ds
  \right\vert
  \\
  & \leq 4\sqrt{\gamma}
  \int_{\frac{3}{5\gamma+3}}^{\frac{5}{3\gamma+5}}
  \int_0^s
  \int_{\vert y \vert \leq 4(R+1)/\sqrt{\gamma}}
  \left\vert u(y,s') 
    \overline{\nabla_A u(y,s')} \right\vert dyds'ds
  \\
  & \leq
  4 \sqrt{\gamma} \left(\frac{5}{3\gamma+5} - \frac{3}{5\gamma
      +3}\right)
  \int_0^{\frac{5}{3\gamma+5}}
  \int_{\vert y \vert  \leq 4(R+1)/\sqrt{\gamma}}
  \left\vert u(y,s') 
    \overline{\nabla_A u(y,s')} \right\vert dyds'
  \\
  & \leq
  \frac{8}{\sqrt{\gamma}} \frac{5}{3\gamma+5}
    \sup_{s' \in [0,1]}
    \int_{\vert y \vert  \leq 4(R+1)/\sqrt{\gamma}}
  \left\vert u(y,s') 
    \overline{\nabla_A u(y,s')} \right\vert dy
  \\ 
  & \leq \frac{16}{\gamma^{3/2}}
  \sup_{s' \in [0,1]}
  \int_{\vert y \vert  \leq 4(R+1)/\sqrt{\gamma}}
  \left\vert u(y,s') 
    \overline{\nabla_A u(y,s')} \right\vert dy
    \\
  & \leq \frac{8}{\gamma^{3/2}}
  \sup_{s' \in [0,1]}
  \int_{\vert y \vert  \leq R_1}
  \left( \abs{u(y,s')}^2 +
    \abs*{\nabla_A u(y,s')}^2 \right) dy = \frac{8}{\gamma^{3/2}} E_u^2.
  \end{split}
\end{equation}
To estimate $E_2$, we reason as above and thanks to \eqref{eq:V.bound},
\eqref{eq:norm.bound} and \eqref{eq:gamma.big} we get
\begin{equation}\label{eq:E2}
  \begin{split}
  E_2  &  \leq \frac{4 M_V}{\gamma} \int_0^{\frac{5}{3\gamma+5}}
  \int_{\vert y \vert \leq 4(R+1)/\sqrt{\gamma}}
  |u(y,s')|^2 dyds' 
  \\ & \leq \frac{8 M_V}{\gamma^2} \sup_{s' \in [0,1]}
  \int_{\vert y \vert \leq 4(R+1)/\sqrt{\gamma}}
    |u(y,s')|^2 dy
    \leq  \frac{8 M_V E_u^2}{\gamma^2}.
  \end{split}
\end{equation}
Thanks to \eqref{eq:gamma.big}, from \eqref{eq:carleman.from.below}--\eqref{eq:E2}
we conclude that
\begin{equation}\label{eq:from.below.final}
  \norm*{e^{\tau\left| \frac{x}{R}+\varphi(t)e_1\right|^2}
  g(x,t)}_{L^2(\R^n \times [0,1])}^2
\geq   \frac{e^{8\tau}}{2^6} M_u^2.
\end{equation}

We estimate from above the right hand side of
\eqref{eq:carleman.proof}: from \eqref{3eq:2appell} and \eqref{eq:defn.g}
we have
\begin{align*}
  (i&\partial_t+\Delta_{\widetilde A} )g (x,t) = -\widetilde{V}(x,t) g(x,t) % \theta_R(x) \eta \left(\frac{x}{R}+
  % \varphi (t) e_1 \right)  \tilde{V}(x,t)
  % v(x,t)
  \\
    &+\theta_R(x) \Big[i\varphi'(t)
      \partial_{x_1} \eta\left(\tfrac{x}{R}+
      \varphi (t) e_1 \right) v(x,t)
      +
      \tfrac{1}{R^{2}} \left(\Delta \eta \left(\tfrac{x}{R} +
      \varphi (t) e_1 \right)\right) v(x,t)
      + \tfrac{2}{R} \nabla \eta\left(\tfrac{x}{R} +
      \varphi (t) e_1 \right) \cdot
      \nabla_{\widetilde A}
      v(x,t) \Big] \\
    & + \eta\left(\tfrac{x}{R}+
      \varphi (t) e_1 \right)
      \left[(\Delta \theta_R(x))  v(x,t) + 2\nabla\theta_R(x) \cdot
      \nabla_{\widetilde A}
      v(x,t)\right] \\
    & + \tfrac{2}{R} \nabla \theta_R(x) \cdot \nabla \eta \left(\tfrac{x}{R} +
      \varphi (t) e_1 \right) v(x,t) \\ 
    & =: F_1(x,t)+F_2(x,t)+F_3(x,t) + F_4(x,t).
\end{align*}
Consequently,
\begin{equation}\label{eq:from.above}
    \norm*{ e^{\tau\left| \frac{x}{R}+\varphi(t)e_1\right|^2}
      (i\partial_t + \Delta_{\widetilde{A}})g(x,t)}_{L^2(\R^n \times [0,1])}
    \leq
    \sum_{i=1}^4
        \norm*{ e^{\tau\left| \frac{x}{R}+\varphi(t)e_1\right|^2}
          F_i(x,t)}_{L^2(\R^n \times [0,1])}.
\end{equation}
We estimate separately the terms at right hand side of the previous inequality.
It is useful to observe that
\begin{equation}\label{eq:for.change.of.variables.3}  
  \frac{1}{\sqrt{\gamma}} \leq \alpha(t) \leq \frac{4}{\sqrt{\gamma}},
  \quad
  0\leq \beta(t) % =\frac{1}{1-t+\gamma^{-1}t}-\frac{1}{\gamma(1-t)+t}
    \leq
  \frac{1}{1-t+\gamma^{-1}t} \leq 4,
  \quad \text{ for all } t \in \left[\frac14,\frac34\right],
\end{equation}
Thanks to \eqref{eq:V.bound}, \eqref{eq:gamma.big}, \eqref{eq:A.V.appell}, \eqref{eq:supp.g}
and \eqref{eq:for.change.of.variables.3}, we have
$\norm{\widetilde{V}}_{L^\infty(\supp g)} \leq
16 M_V/\gamma$ and  
\begin{equation}\label{eq:from.above.1}
  \begin{split}
        \norm*{ e^{\tau\left| \frac{x}{R}+\varphi(t)e_1\right|^2}
          F_1(x,t)}_{L^2(\R^n \times [0,1])} 
        &\leq \frac{16 M_V}{\gamma}
          \norm*{e^{\tau\left| \frac{x}{R}+\varphi(t)e_1\right|^2}
            g(x,t)}_{L^2(\R^n \times [0,1])}
        \\ &\leq
          \frac{2^{16} E_u}{M_u}
          \norm*{e^{\tau\left| \frac{x}{R}+\varphi(t)e_1\right|^2}
            g(x,t)}_{L^2(\R^n \times [0,1])}.
        \end{split}
      \end{equation}
Observe that in the support of $F_2$ we have $\abs*{\frac{x}{R} +
  \varphi(t)e_1} \leq 2$, and thanks to \eqref{eq:gamma.big}, % $R>2$
\eqref{eq:control.derivatives.theta}--%
% \eqref{eq:control.derivatives.eta},
\eqref{eq:control.derivatives.varphi} and \eqref{eq:supp.g}
we estimate
\begin{equation} \label{eq:F_2}
  \norm*{ e^{\tau\left| \frac{x}{R}+\varphi(t)e_1\right|^2}
    F_2(x,t)}_{L^2(\R^n \times [0,1])}^2
  \leq 2^{14} e^{8\tau}
  \int_{1/4}^{3/4}
  \int_{\vert x \vert \leq R+1} \left(\abs{v(x,t)}^2
  + \frac{1}{R^{2}} \abs{\nabla_{\widetilde A} v}^2\right)dxdt
  = F_{21} + F_{22}.
\end{equation}
We use again the change of variables in
\eqref{eq:change.of.variables}: we observe that
\begin{equation}
  \label{eq:for.change.of.variables.4}
  \frac{\gamma}{16}\leq \frac{dt}{ds}(s) % =\frac{\gamma}{(1+s\gamma-s)^2} 
  \leq \gamma,
  \quad \text{ for all }s\in 
  \left[\frac{1}{3\gamma+1}, \frac{3}{\gamma+3}\right]
  = s\left(\left[\frac14,\frac34\right]\right).
\end{equation}
Thanks to \eqref{eq:for.change.of.variables.4}, we have that
\begin{equation}\label{eq:join.1}
  F_{21} 
  %  2^{14}e^{8\tau}
  % \int_{1/4}^{3/4}
  % \int_{\vert x \vert \leq R+1} \abs{v(x,t)}^2  dxdt
  \leq 2^{14}e^{8\tau} \gamma
  \int_{\frac{1}{3\gamma+1}}^{\frac{3}{\gamma+3}}
  \int_{\abs{y} \leq \alpha(t(s)) (R+1)}
  \abs{u(y,s)}^2 dy ds.
\end{equation}
Thanks to \eqref{eq:defn.v}, \eqref{eq:for.change.of.variables.3} and \eqref{eq:for.change.of.variables.4},
\begin{equation*}
  \begin{split}
    F_{22} % & = 2^{14}e^{8\tau}\int_{1/4}^{3/4}\int_{\vert x\vert\leq R+1}
    % \frac{1}{R^{2}} \abs{\nabla_{\widetilde A} v(x,t)}^2dxdt \\
    & = {2^{14} e^{8\tau}}\int_{1/4}^{3/4}\int_{\vert x\vert\leq R+1}
    \frac{\alpha(t)^n}{R^{2}} \abs*{\alpha(t)(\nabla_A u)(\alpha(t)x,
      s(t))-\frac{i}{2}\beta(t)x \, u (\alpha(t)x, s(t))}^2 dxdt \\
    & \leq  2^{14} e^{8\tau}  \gamma
    \int_{\frac{1}{3\gamma+1}}^{\frac{3}{\gamma+3}}
    \int_{|y|\leq \alpha(t(s))(R+1)}
    \frac{1}{R^2}\abs*{\alpha(t(s))\nabla_A u(y,s) -
      \frac{i \beta(t(s)) y}{2\alpha(t(s))} u(y,s)}^2 dyds \\
    & \leq 2^{14} e^{8\tau} \gamma
    \int_{\frac{1}{3\gamma+1}}^{\frac{3}{\gamma+3}}
    \int_{|y|\leq \alpha(t(s))(R+1)}
    \left( \frac{32}{\gamma R^2}|\nabla_A u(y,s)|^2 +
      8\left(1+\frac{1}{R}\right)^2|u(y,s)|^2 \right) dyds.
  \end{split}
\end{equation*}
Thanks to \eqref{eq:gamma.big}, from the last inequality we conclude that
\begin{equation}\label{eq:join.2}
  F_{22} \leq 2^{19} e^{8\tau} \gamma
  \int_{\frac{1}{3\gamma+1}}^{\frac{3}{\gamma+3}}
  \int_{|y|\leq \alpha(t(s))(R+1)}
  \left( |u(y,s)|^2 + |\nabla_A u(y,s)|^2 \right) dyds.
\end{equation}
From \eqref{eq:norm.bound}, \eqref{eq:gamma.big}, % \eqref{eq:for.change.of.variables.3},
\eqref{eq:join.1},
\eqref{eq:join.2} and
since $\abs{[\frac{1}{3\gamma+1},\frac{3}{\gamma+3}]} \leq
\frac{4}{\gamma}$, we get
\begin{equation}\label{eq:from.above.2}
   \norm*{ e^{\tau\left| \frac{x}{R}+\varphi(t)e_1\right|^2}
    F_2(x,t)}_{L^2(\R^n \times [0,1])}^2
 \leq 2^{22} e^{8\tau}
  \sup_{s \in [0,1]}
  \int_{|y|\leq R_1}
  \left( |u(y,s)|^2 + |\nabla_A u(y,s)|^2 \right) dy
  \leq
  2^{22} e^{8\tau} E_u^2.
\end{equation}
We treat now the term in $F_3$. We observe that in its support we
have $R \leq \abs{x}\leq R+1$ and $|\frac{x}{R}+\varphi(t)e_1|\leq 6$
thanks to \eqref{eq:control.derivatives.varphi} and since $R>2$. 
Thanks to \eqref{eq:control.derivatives.theta} and \eqref{eq:control.derivatives.eta} we have
\begin{equation*}
  \norm*{ e^{\tau\left| \frac{x}{R}+\varphi(t)e_1\right|^2}
    F_3(x,t)}_{L^2(\R^n \times [0,1])}^2
  \leq
  8 e^{72\tau}\int_{1/4}^{3/4}
  \int_{R \leq \vert x \vert \leq R+1}
  (\abs{v(x,t)}^2 + \abs{\nabla_{\widetilde A} v(x,t)}^2)dxdt
  =: F_{31} + F_{32}.
\end{equation*}
Using again the change of coordinates \eqref{eq:change.of.variables}
and reasoning as in the estimate \eqref{eq:join.1}, we have
\begin{equation}\label{eq:F31}
  F_{31} \leq 
  8 e^{72\tau} \gamma
  \int_{\frac{1}{3\gamma+1}}^{\frac{3}{\gamma+3}}
  \int_{R \leq \frac{\abs{y}}{\alpha(t(s))} \leq  R+1}
  \abs{u(y,s)}^2 dy ds.
\end{equation}
Since $R^2 = R_0^2 \gamma$, $R>2$, thanks to \eqref{eq:defn.v}, \eqref{eq:for.change.of.variables.3},
\eqref{eq:for.change.of.variables.4} and  reasoning as in the estimate
\eqref{eq:join.2} we get
\begin{equation}\label{eq:F32}
  \begin{split}
  F_{32} & \leq
    8 e^{72\tau} \gamma^2 R_0^2 
  \int_{\frac{1}{3\gamma+1}}^{\frac{3}{\gamma+3}}
  \int_{R \leq \frac{\abs{y}}{\alpha(t(s))} \leq  R+1}
  \left(\frac{32}{\gamma R^2} \abs{\nabla_{A} u(y,s)}^2 +
    8 \left(1+\frac{1}{R}\right)^2 \abs{u(y,s)}^2\right) \, dyds
  \\
  & \leq
      2^{8} e^{72\tau} \gamma^2 R_0^2 
  \int_{\frac{1}{3\gamma+1}}^{\frac{3}{\gamma+3}}
  \int_{R \leq \frac{\abs{y}}{\alpha(t(s))} \leq  R+1}
  \left(  \abs{u(y,s)}^2 + \frac{1}{\gamma} \abs{\nabla_{A} u(y,s)}^2
    \right) \, dyds.  
\end{split}
\end{equation}
From \eqref{eq:F31} and \eqref{eq:F32}, we have 
\begin{equation*}
    \norm*{ e^{\tau\left| \frac{x}{R}+\varphi(t)e_1\right|^2}
    F_3(x,t)}_{L^2(\R^n \times [0,1])}^2
  \leq
  2^{9}   e^{72\tau} \gamma^2 R_0^2   \int_{\frac{1}{3\gamma+1}}^{\frac{3}{\gamma+3}}
  \int_{ R\leq \frac{|y|}{\alpha(t(s))} \leq R+1}
  \left( |u(y,s)|^2+\gamma^{-1}|\nabla_{A} u(y,s)|^2\right) dyds.
\end{equation*}
The length of the above space integration region is $\alpha(t(s))$. In
order to write it in terms of $\gamma$,
we see by \eqref{eq:change.of.variables} and
\eqref{eq:for.change.of.variables.3}, and since
$\alpha(t(s))\sqrt{\gamma} = 1 + s(\gamma -1)$ that
\begin{equation*}
  \left\{(y,s) \, \middle| \, \alpha(t(s))R\leq |y| \leq \alpha(t(s))(R+1), s\in
    \left[{\frac{1}{3\gamma+1}},{\frac{3}{\gamma+3}}\right]  \right\}
  \subset
  \left\{ (y,s) \,  \middle| \, \big\vert \abs{y} -R_0 -R_0s\gamma\big\vert
      \leq \frac{4(R_0+1)}{\sqrt{\gamma}} \right\},
  \end{equation*}
  %  Follows from
  %   \begin{equation*}
  %     \begin{split}
  %     &|y| - R_0 -R_0 s \gamma
  %     \geq \alpha(t(s))R - R_0 - R_0 s \gamma
  %     =
  %     (1 +s\gamma -s) R_0 - R_0(1+s\gamma) \\
  %     & = -s R_0 \geq
  %     -\frac{3}{\gamma+3} R_0 \geq -\frac{4}{\sqrt{\gamma}} R_0
  %     \geq -\frac{4}{\sqrt{\gamma}} (R_0+1)
  %   \end{split}
  % \end{equation*}
  % and
  % \begin{equation*}
  %   \begin{split}
  %     &|y| -R_0 -R_0 s \gamma \leq
  %     \alpha(t(s))R + \alpha(t(s)) - R_0 - R_0 s \gamma
  %     \leq 
  %     \alpha(t(s)) - s R_0 \leq \alpha(t(s))
  %     \\ & \leq \frac{1 + s(\gamma -1)}{\sqrt{\gamma}} \leq
  %     \frac{1}{\sqrt{\gamma}} +
  %     \frac{3}{\sqrt{\gamma}}\frac{\gamma-1}{\gamma+3}
  %     \leq \frac{4}{\sqrt{\gamma}} \leq \frac{4}{\sqrt{\gamma}}(R_0+1)
  %   \end{split}
  % \end{equation*}
therefore
\begin{equation}\label{eq:from.above.3}
  \begin{split}
     & \norm*{ e^{\tau\left| \frac{x}{R}+\varphi(t)e_1\right|^2}
       F_3(x,t)}_{L^2(\R^n \times [0,1])}^2
     \\ & 
\leq  2^9   e^{72\tau} \gamma^2 R_0^2
    \int_{\frac{1}{3\gamma+1}}^{\frac{3}{\gamma+3}}
    \int_{||y|-R_0-R_0s\gamma|<\frac{4(R_0+1)}{\sqrt{\gamma}}}
    \left(\abs{u(y,s)}^2 + \gamma^{-1} \abs{\nabla_A u (y,s)}^2
    \right) dyds.
  \end{split}
\end{equation}
Finally, we treat the term in $F_4$: we reason analogously as done in
the estimates of the terms in $F_3$ and $F_4$. Thanks to
\eqref{eq:norm.bound}, \eqref{eq:gamma.big},
\eqref{eq:control.derivatives.theta} and
\eqref{eq:control.derivatives.eta} we have
\begin{equation}\label{eq:from.above.4}
  \begin{split}
    & \norm*{ e^{\tau\left| \frac{x}{R}+\varphi(t)e_1\right|^2}
      F_4(x,t)}_{L^2(\R^n \times [0,1])}^2
    \\ &
    \leq
    4   e^{8 \tau} 
    \int_{1/4}^{3/4} \int_{R\leq \abs{x}\leq R+1} \abs{v(x,t)}^2 \, dxdt
    \leq
    4 e^{8 \tau} \gamma
    \int_{\frac{1}{3\gamma+1}}^{\frac{3}{\gamma+3}}
    \int_{R\leq \frac{\abs{y}}{\alpha(t(s))} \leq R+1}
    \abs{u(y,s)}^2 \,dyds
    \\ &
    \leq
    2^4 e^{8\tau} \sup_{s\in [0,1]} \int_{\abs{y}\leq R_1}
    \abs{u(y,s)}^2\,dyds \leq 2^4 e^{8\tau} E_u^2.
  \end{split}
\end{equation}

Gathering \eqref{eq:carleman.proof}, \eqref{eq:from.above}, \eqref{eq:from.above.1},
\eqref{eq:from.above.2}, \eqref{eq:from.above.3} and \eqref{eq:from.above.4},
we conclude that 
\begin{equation}\label{eq:from.above.final}
  \begin{split}
  &   \frac{\tau^{3/2}}{cR^2}
  \norm*{e^{\tau\left| \frac{x}{R}+\varphi(t)e_1\right|^2}
    g(x,t)}_{L^2(\R^n \times [0,1])} \leq
  \norm*{ e^{\tau\left| \frac{x}{R}+\varphi(t)e_1\right|^2}
    (i\partial_t + \Delta_{\widetilde{A}})g(x,t)}_{L^2(\R^n \times [0,1])}
  \\ & \leq
    \frac{2^{16} E_u}{M_u}
       \norm*{e^{\tau\left| \frac{x}{R}+\varphi(t)e_1\right|^2}
    g(x,t)}_{L^2(\R^n \times [0,1])}
  +
  2^{12} e^{4\tau} E_u
  \\ & \quad +
  \left(2^9   e^{72 \tau} \gamma^2 R_0^2
    \int_{\frac{1}{3\gamma+1}}^{\frac{3}{\gamma+3}}
    \int_{||y|-R_0-R_0s\gamma|<\frac{4(R_0+1)}{\sqrt{\gamma}}}
    \left(\abs{u(y,s)}^2 + \gamma^{-1} \abs{\nabla_A u (y,s)}^2\right)
    dyds\right)^{\frac{1}{2}}.
\end{split}
\end{equation}

We set $\tau := 64 c R^2$. Thanks to \eqref{eq:gamma.big}, we have 
\begin{equation*}
   \frac{2^{16}E_u}{M_u} \leq \frac{\tau^{3/2}}{2cR^2}.
\end{equation*}
By \eqref{eq:from.below.final}, from \eqref{eq:from.above.final} we get
\begin{equation}
  \begin{split}
  \label{eq:1}
  &2^5 e^{4\tau}\sqrt{c}RM_u = \frac{\tau^{\frac32}}{c R^2} \frac{e^{4\tau}M_u}{2^4}
  \\ & \leq
  2^{12} e^{4\tau} E_u
  +
  \left(2^9   e^{72 \tau} \gamma^2 R_0^2
    \int_{\frac{1}{3\gamma+1}}^{\frac{3}{\gamma+3}}
    \int_{||y|-R_0-R_0s\gamma|<\frac{4(R_0+1)}{\sqrt{\gamma}}}
    \left(\abs{u(y,s)}^2 + \gamma^{-1} \abs{\nabla_A u (y,s)}^2\right)
    dyds\right)^{\frac{1}{2}}.
\end{split}
\end{equation}
Thanks to \eqref{eq:gamma.big}, we have
\begin{equation*}
  2^{12} e^{4\tau} E_u \leq  2^4 e^{4\tau}\sqrt{c}RM_u,
\end{equation*}
so we conclude that 
\begin{equation*}
  2^8 e^{8\tau} c R^2 M_u^2  \leq
  2^9   e^{72 \tau} \gamma^2 R_0^2
    \int_{\frac{1}{3\gamma+1}}^{\frac{3}{\gamma+3}}
    \int_{||y|-R_0-R_0s\gamma|<\frac{4(R_0+1)}{\sqrt{\gamma}}}
    \left(\abs{u(y,s)}^2 + \gamma^{-1} \abs{\nabla_A u (y,s)}^2\right)
    dyds,
\end{equation*}
that is to say
\begin{equation*}
  M_u^2  \leq
  \frac{2   e^{2^{12} c R_0^2 \gamma}}{c}  \gamma
    \int_{\frac{1}{3\gamma+1}}^{\frac{3}{\gamma+3}}
    \int_{||y|-R_0-R_0s\gamma|<\frac{4(R_0+1)}{\sqrt{\gamma}}}
    \left(\abs{u(y,s)}^2 + \gamma^{-1} \abs{\nabla_A u (y,s)}^2\right)
    dyds.
\end{equation*}
Consequently, for $C = C(M_B) = \max(2^{12}c, 2 c^{-1}) >0$ we have
\begin{equation}\label{eq:almost.the.end}
  M_u^2  \leq
  C e^{ C R_0^2 \gamma}  \gamma
    \int_{\frac{1}{4\gamma}}^{\frac{3}{\gamma}}
    \int_{||y|-R_0-R_0s\gamma|<\frac{4(R_0+1)}{\sqrt{\gamma}}}
    \left(\abs{u(y,s)}^2 + \gamma^{-1} \abs{\nabla_A u (y,s)}^2\right)
    dyds.
\end{equation}
We let $t:=\gamma^{-1}$: from \eqref{eq:almost.the.end} we get that
\eqref{eq:thesis} holds for all $0< t< t^\ast := (\gamma^\ast)^{-1}$ and
$\rho=R_0$. In order to complete the proof for any $\rho\in[R_0, R_1]$, it is sufficient to repeat the same argument as above, choosing $R=\rho\sqrt\gamma$ in \eqref{eq:erre}.

\bibliographystyle{siam}

\end{document}